\documentclass[12pt,oneside,english]{amsart}
\usepackage{lmodern}
\usepackage[T1]{fontenc}
\usepackage[latin9]{inputenc}
\usepackage{color}
\usepackage{babel}
\usepackage{mathtools}
\usepackage{amsthm}
\usepackage{amssymb}
\usepackage{setspace}
\usepackage[pdfusetitle,
 bookmarks=true,bookmarksnumbered=false,bookmarksopen=false,
 breaklinks=false,pdfborder={0 0 0},pdfborderstyle={},backref=false,colorlinks=true]
 {hyperref}

\makeatletter
\numberwithin{equation}{section}
\numberwithin{figure}{section}
\theoremstyle{plain}
\newtheorem*{thm*}{\protect\theoremname}
\theoremstyle{plain}
\newtheorem{thm}{\protect\theoremname}[section]
\theoremstyle{plain}
\newtheorem{lem}[thm]{\protect\lemmaname}
\theoremstyle{plain}
\newtheorem{cor}[thm]{\protect\corollaryname}
\theoremstyle{remark}
\newtheorem{rem}[thm]{\protect\remarkname}
\theoremstyle{plain}
\newtheorem{prop}[thm]{\protect\propositionname}

\hypersetup{
    colorlinks,
    linkcolor={blue},
    citecolor={blue},
    urlcolor={blue}
}
\usepackage[a4paper,margin=3cm,includehead,includefoot]{geometry}
\usepackage{cleveref}

\makeatother

\providecommand{\corollaryname}{Corollary}
\providecommand{\lemmaname}{Lemma}
\providecommand{\propositionname}{Proposition}
\providecommand{\remarkname}{Remark}
\providecommand{\theoremname}{Theorem}

\begin{document}
\global\long\def\CC{\mathbb{C}}%
\global\long\def\norm#1{\left\Vert #1\right\Vert }%
\global\long\def\sq{\subseteq}%
\global\long\def\HHH{\mathcal{H}}%
\global\long\def\st{\;|\;}%
\global\long\def\inj{\hookrightarrow}%
\global\long\def\ph{\varphi}%
\global\long\def\RR{\mathbb{R}}%
\global\long\def\NN{\mathbb{N}}%
\global\long\def\U{\mathcal{U}}%
\global\long\def\lim#1{\underset{#1}{lim}}%
\global\long\def\ZZ{\mathbb{Z}}%
\global\long\def\FF{\mathbb{F}}%
\global\long\def\frakm{\mathfrak{m}}%
\global\long\def\lto{\longrightarrow}%
\global\long\def\conv#1{\underset{#1}{\lto}}%
\global\long\def\mf#1{\mathfrak{#1}}%

\title{Uniform rank-metric stability of Lie algebras and groups}
\author{Benjamin Bachner}
\begin{abstract}
We study uniform stability of discrete groups, Lie groups and Lie
algebras in the rank metric, and the connections between uniform stability
of these objects. We prove that semisimple Lie algebras are far from
being flexibly $\CC$-stable, and that semisimple Lie groups and lattices
in semisimple Lie groups of higher rank are not strictly $\CC$-stable.
Furthermore, we prove that free groups are not uniformly flexibly
$F$-stable over any field $F$.
\end{abstract}

\maketitle

\section{Introduction}

The notion of homomorphism stability is the general problem of whether
almost homomorphisms are close to genuine homomorphisms, and is usually
referred to as Ulam stability \cite{Ulam}. This property depends
on the objects in question (usually groups or algebras), and on the
precise meaning of almost homomorphisms and distances between them,
as well as the chosen metric. A theorem of Kazhdan gives such a result
for unitary almost representations in the operator norm:
\begin{thm*}
(Kazhdan 1982 \cite{KazhdanStability})

Let $\Gamma$ be an amenable group and $\ph:\Gamma\to U\left(\mathcal{H}\right)$
be a function into the unitary operators on a Hilbert space $\mathcal{H}$.
If $\norm{\ph\left(\gamma_{1}\gamma_{2}\right)-\ph\left(\gamma_{1}\right)\ph\left(\gamma_{2}\right)}_{op}\le\delta$
for all $\gamma_{1},\gamma_{2}\in\Gamma$ and $\delta\le\frac{1}{200}$,
then there exists a representation $\psi:\Gamma\to U\left(\mathcal{H}\right)$
such that $\norm{\ph\left(\gamma\right)-\psi\left(\gamma\right)}_{op}\le2\delta$
for all $\gamma\in\Gamma$.
\end{thm*}
In general, we say that the group $\Gamma$ is \emph{uniformly stable}
in the operator norm, if for any $\varepsilon>0$ there exists a $\delta>0$
such that for any $n\in\NN$ and any map $\ph:\Gamma\to U\left(n\right)$
satisfying 
\[
\norm{\ph\left(\gamma_{1}\gamma_{2}\right)-\ph\left(\gamma_{1}\right)\ph\left(\gamma_{2}\right)}_{op}\le\delta\quad\forall\gamma_{1},\gamma_{2}\in\Gamma,
\]
there exists a genuine representation $\psi:\Gamma\to U\left(n\right)$
such that $\norm{\ph\left(\gamma\right)-\psi\left(\gamma\right)}_{op}\le\varepsilon$
for all $\gamma\in\Gamma$\footnote{This property was called ``strong Ulam stability'' in \cite{BurgerOzawaThom}.}.

Burger, Ozawa and Thom \cite{BurgerOzawaThom} have proved that the
groups $SL_{n}\left(\mathcal{O}_{S}\right)$ are uniformly stable
in the operator norm, where $n\ge3$ and $\mathcal{O}$ is the ring
of integers of a number field and $\mathcal{O}_{S}$ some localization
of $\mathcal{O}$. Glebsky, Lubotzky, Monod and Rangarajan \cite{AsymptoticCoho}
have proved uniform stability in $U\left(n\right)$ with respect to
arbitrary submultiplicative norms for all amenable groups (generalizing
Kazhdan's theorem for the case of finite dimensional Hilbert spaces)
and for a large class of lattices in higher rank semisimple Lie groups
using vanishing properties of asymptotic cohomology. Rolli \cite{Rolli}
proved a negative result for nonabelian free groups by constructing
maps $\ph:\FF_{m}\to U\left(n\right)$ with $\norm{\ph\left(\gamma_{1}\gamma_{2}\right)-\ph\left(\gamma_{1}\right)\ph\left(\gamma_{2}\right)}_{op}\le\delta$
for any $\gamma_{1},\gamma_{2}\in\Gamma$ and arbitrarily small $\delta$,
that are bounded from all representations $\psi:\Gamma\to U\left(n\right)$.

\medskip{}

Following Arzhantseva and P\u{a}unescu \cite{AP_LinearSofic}, we
consider a similar problem with the unitary groups replaced by $GL_{n}\left(F\right)$
(for some field $F$), endowed with the \emph{normalized rank metric:
\[
rk\left(A-B\right)\vcentcolon=\frac{1}{n}rank\left(A-B\right)\qquad\forall A,B\in GL_{n}\left(F\right).
\]
}This metric can be seen as a generalization of the normalized Hamming
metric from finite permutation groups to general linear groups, see
\cite{AP_LinearSofic} for details.

\medskip{}

For a group $\Gamma$, an\emph{ $\varepsilon$-representation} is
a function $\ph:\Gamma\to GL_{n}\left(F\right)$ satisfying 
\[
\begin{aligned}rk\left(\ph\left(\gamma_{1}\right)\ph\left(\gamma_{2}\right)-\ph\left(\gamma_{1}\gamma_{2}\right)\right)\le\varepsilon\quad\forall\gamma_{1},\gamma_{2}\in\Gamma\end{aligned}
.
\]

We say that the group $\Gamma$ is \emph{uniformly $F$-stable} if
for any $\varepsilon>0$, there exists a $\delta>0$ such that for
any $n\in\NN$ and any $\delta$-representation $\ph:\Gamma\to GL_{n}\left(F\right)$,
there exists a representation $\psi:\Gamma\to GL_{n}\left(F\right)$
such that $rk\left(\ph\left(\gamma\right)-\psi\left(\gamma\right)\right)\le\varepsilon\quad\forall\gamma\in\Gamma$. 

\medskip{}

Our first main result shows that many lattices in higher rank semisimple
Lie groups are not uniformly $\CC$-stable, as opposed to submultiplicative
norms as in \cite{AsymptoticCoho}.
\begin{thm}
\label{thm:Lattices}Let $G$ be a connected linear semisimple real
Lie group, defined over $\mathbb{Q}$ and of rank at least $2$ with
no compact factors, and $\Gamma\le G$ an irreducible lattice. Then
$\Gamma$ is not uniformly $\CC$-stable.
\end{thm}

\medskip{}
Similarly to stability in permutations \cite{BeckerChapman} and in
the normalized Hilbert-Schmidt norm \cite{DeChiffreOzawaThom}, we
consider also a notion of \emph{flexible stability}, where the homomorphism
$\psi:\Gamma\to GL_{N}\left(F\right)$ is allowed to take values in
matrices of larger dimension $N\ge n$, and $\ph\left(\Gamma\right)$
is embedded in $M_{N}\left(F\right)$ as a corner, see Section \ref{subsec:Stability}.
We will usually refer to stability in the normal sense (where an increase
in dimension is not allowed) as \emph{strict stability}. 

The question of whether lattices as in Theorem \ref{thm:Lattices}
are \emph{flexibly} uniformly stable remains open, and we discuss
this in Section \ref{subsec:Remarks}. 

\medskip{}

Our second main result concerns flexible stability of semisimple Lie
algebras, which will be shown to be connected to the stability of
semisimple Lie groups and their lattices. For a Lie algebra $\mf g$,
an $\varepsilon$-representation is a linear map $\ph:\mf g\to\mf{gl}_{n}\left(F\right)$
satisfying 
\[
rk\left(\left[\ph\left(x\right),\ph\left(y\right)\right]-\ph\left(\left[x,y\right]\right)\right)\le\varepsilon\quad\forall x,y\in\mf g,
\]
and $\mf g$ is uniformly $F$-stable if for any $\varepsilon>0$,
there exists a $\delta>0$ such that for any $n\in\NN$ and any $\delta$-representation
$\ph:\mf g\to\mf{gl}_{n}\left(F\right)$, there exists a representation
$\psi:\mf g\to\mf{gl}_{n}\left(F\right)$ such that $rk\left(\ph\left(x\right)-\psi\left(x\right)\right)\le\varepsilon\quad\forall x\in\mf g$. 

The following result was proven independently in \cite{BauerBlacharGreenfeld}.
\begin{thm}
Let $\mf g$ be a finite-dimensional semisimple real Lie algebra.
Then $\mf g$ is not flexibly uniformly $\CC$-stable.
\end{thm}

Finally, we provide counterexamples for flexible stablity of groups,
by a construction in the spirit of Rolli \cite{Rolli} who proved
that free groups are not stable in the operator norm, in addition
to the proof in \cite{BeckerChapman} that free groups are not flexibly
stable in permutations.
\begin{thm}
\label{thm:free}Let $\Gamma$ be a group that surjects onto the free
group $\FF_{2}$ of rank $2$. Then $\Gamma$ is not flexibly $F$-stable
for any field $F$.
\end{thm}

There are various results on uniform stability of discrete groups
with respect to different classes of groups and metrics. De-Chiffre,
Ozawa and Thom \cite{DeChiffreOzawaThom} have generalized a result
by Gowers and Hatami \cite{GowersHatami}, proving that all amenable
groups are uniformly flexibly stable in the Hilbert-Schmidt norm.
In contrast, Akhtiamov and Dogon \cite{AkhtiamovDogon} have proved
that for $\Gamma$ a finitely generated residually finite group, \emph{strict}
uniform Hilbert-Schmidt stability is equivalent to $\Gamma$ being
virtually abelian. 

Becker and Chapman \cite{BeckerChapman} have proved that amenable
groups are uniformly flexibly stable in permutations, and that in
many cases the flexibility is necessary; For any group $\Gamma$ with
finite quotients of unbounded size, $\Gamma$ is not strictly stable.
Additionally, they have proved that the groups $SL_{r}\left(\mathcal{O}_{K}\right)$
are uniformly flexibly stable in permutations, where $r\ge3$ and
$\mathcal{O}_{K}$ the ring of integers in a number field $K$. 

\medskip{}

In general, representations of a group into $GL_{n}\left(\CC\right)$
tend to be much less understood than unitary representations. For
this reason, we focus in this paper on a class of groups whose (finite
dimensional) representation theory is relatively well behaved, namely,
irreducible lattices in higher rank semisimple Lie groups. We consider
compressions of representations of $\Gamma\le G$ coming from the
Lie group $G$, and taking differentials at the identity we obtain
almost representations of its Lie algebra $\mf g$. Thus, the notion
of stability of Lie algebras arises naturally in this context, and
our main focus will be in the case where the Lie group/algebra is
semisimple. In Theorem \ref{thm:thm9}, we use central characters
of representations of $\mf g$ to prove the existence of an uncountable
family of ``highest weight'' $\varepsilon$-representations of any
semisimple Lie algebra, that are nonequivalent in a suitable sense.
In particular, this reestablishes the fact that they are not flexibly
stable. We then use the connection between stability of Lie groups
and stability of their Lie algebras, as well as rigidity properties
for lattices in semisimple Lie groups to deduce Theorem \ref{thm:Lattices}.

\medskip{}

There is an important related notion of \emph{pointwise stability,
}where for a finitely presented group/algebra, an $\varepsilon$-homomorphism
is a map $\ph:S\to GL_{n}\left(F\right)$ ($\mf{gl}_{n}\left(F\right)$
in the case of algebras) from a set of generators $S$ such that the
finite set of relations $R$ are $\varepsilon$-satisfied in $GL_{n}\left(F\right)$
($\mf{gl}_{n}\left(F\right)$). The group/algebra is then \emph{pointwise
$F$-stable} if for any $\varepsilon>0$, there exists a $\delta>0$
such that for any $\delta$-representation $\ph:S\to GL_{n}\left(F\right)$
($\mf{gl}_{n}\left(F\right)$), there exists a representation $\psi:S\to GL_{n}\left(F\right)$
($\mf{gl}_{n}\left(F\right)$) such that $rk\left(\ph\left(s\right)-\psi\left(s\right)\right)\le\varepsilon,\;\forall s\in S$.
Flexible stability is defined analogously to uniform stability. We
refer to \cite{BauerBlacharGreenfeld} for a precise definition and
examples, and note that the term stability there corresponds to what
we call flexible stability in this paper.

Some form of pointwise $\CC$-stability of $\ZZ^{k}$ were established
by Elek and Grabowski in \cite{ElekGrabowski}, where the $\varepsilon$-representations
take values in unitary/self-adjoint matrices. This can be thought
of as $k$ matrices that pairwise almost commute, and pointwise stability
means that they are close to $k$ matrices that truly commute. 

Bauer, Blacher and Greenfeld \cite{BauerBlacharGreenfeld} have studied
\emph{pointwise} $\CC$-stability of general polynomial equations.
This includes finitely presented groups, associative algebras and
Lie algebras. They have proved that finite-dimensional associative
algebras are flexibly pointwise stable, but in general are not strictly
pointwise stable. Moreover they have proved that finite-dimensional
Lie algebras are typically not flexibly stable: If $\mf g$ is a finite
dimensional complex Lie algebra that is not solvable or has a nonabelian
nilpotent quotient, then it is not flexibly $\CC$-stable (\cite[Theorem 8.5]{BauerBlacharGreenfeld}).
Thus any nonabelian finite-dimensional complex Lie algebra is not
flexibly $\CC$-stable, and the abelian case is open (and is equivalent
to pointwise stability of $\ZZ^{k}$ as in \cite{ElekGrabowski}).
Although \cite{BauerBlacharGreenfeld} deals with pointwise stability
whereas in this paper we study uniform stability, for finite dimensional
associative/Lie algebras these notions can be shown to be equivalent,
see Section \ref{subsec:Stability}.

\subsection*{Acknowledgements}

I would like to thank my advisor Alex Lubotzky for his guidance and
valuable suggestions. I would also like to thank Alon Dogon and Francesco
Fournier-Facio for helpful remarks and suggestions. This work is part
of the PhD thesis of the author and was supported by the European
Research Council (ERC) under the European Union\textquoteright s Horizon
2020 (N. 882751).

\section{Preliminaries and first observations}

\subsection{\label{subsec:Stability}Uniform Stability}

Let $\Gamma$ be a group and let $\mathcal{G}=\left\{ H_{i}\right\} _{i\in I}$
be a family of groups. Let $\mathcal{F}=\left\{ \ph:\Gamma\to H\st H\in\mathcal{G}\right\} $
be the set of functions from $\Gamma$ to groups in $\mathcal{G}$,
a \emph{defect function $def:\mathcal{F}\to\RR_{\ge0}$ }and a \emph{distance
function} $d:\mathcal{F}\times\mathcal{F}\to\left[0,\infty\right]$.
We say that $\Gamma$ is $\left(\mathcal{G},def,d\right)$\emph{-stable}
if for any $\varepsilon>0$ there exists a $\delta>0$ such that for
any function $\ph\in\mathcal{F}$ with $def\left(\ph\right)\le\delta$
there exists a homomorphism $\psi\in\mathcal{F}$ such that $d\left(\ph,\psi\right)\le\varepsilon$.

In our context, the family $\mathcal{G}$ will be all invertible matrices
over $F$, $\mathcal{GL}\left(F\right)=\left\{ GL_{n}\left(F\right)\right\} _{n\in\NN}$,
and the the (uniform) defect function will be 
\[
def\left(\ph\right)=sup\,\left\{ rk\left(\ph\left(\gamma_{1}\right)\ph\left(\gamma_{2}\right)-\ph\left(\gamma_{1}\gamma_{2}\right)\right)\st\gamma_{1},\gamma_{2}\in\Gamma\right\} .
\]
The distance $d$ with respect to \emph{strict} uniform stability
will be 
\[
d^{strict}\left(\ph,\psi\right)=\begin{cases}
\underset{\gamma\in\Gamma}{sup}\;rk\left(\ph\left(\gamma\right)-\psi\left(\gamma\right)\right) & \ph,\psi:\Gamma\to GL_{m}\left(F\right)\\
\infty & else
\end{cases},
\]
so that the codomains of $\ph$ and $\psi$ must coincide. In the
context of \emph{flexible stability,} we allow taking the distance
between maps with different codomains. We use the notation used in
\cite{ElekGrabowski} and \cite{BauerBlacharGreenfeld}: For $A\in M_{n}\left(F\right)$,
let $\widehat{A}$ be the operator on $\CC^{\oplus\NN}$ which acts
as $A$ on the first $n$ basis vectors and $0$ otherwise. Thus we
may think of $\widehat{A}$ as an infinite matrix that vanishes everywhere
except for the submatrix $\widehat{A}_{\left[n\right],\left[n\right]}$
where it equals $A$. For $A\in GL_{n}\left(F\right)$ and $B\in GL_{m}\left(F\right)$,
define
\[
rk\left(A-B\right)\vcentcolon=\frac{dim\left(im\left(\widehat{A}-\widehat{B}\right)\right)}{min\left(n,m\right)},
\]
And define the flexible distance of $\ph,\psi\in\mathcal{F}$ by
\[
d^{flex}\left(\ph,\psi\right)=\underset{\gamma\in\Gamma}{sup}\,rk\left(\ph\left(\gamma\right)-\psi\left(\gamma\right)\right).
\]

We note that since we require the target matrices to be nonsingular,
we have 
\[
dim\left(im\left(\widehat{A}-\widehat{B}\right)\right)\ge\left|dim\left(im\left(\widehat{A}\right)\right)-dim\left(im\left(\widehat{B}\right)\right)\right|=\left|n-m\right|,
\]
so that $rk\left(A-B\right)$ can be small only if $\frac{\left|n-m\right|}{min\left(n,m\right)}$
is small. Thus, in flexible stability we allow an increase of the
target dimension, but this increase must be small relative to the
dimension.\medskip{}

The group $\Gamma$ is strictly uniformly $F$-stable if it is $\left(\mathcal{GL}\left(F\right),def,d^{strict}\right)$-stable
and flexibly uniformly $F$-stable if it is $\left(\mathcal{GL}\left(F\right),def,d^{flex}\right)$-stable.
\medskip{}

When working with Lie groups, we would like the relevant functions
to be in the associated category, which in this case is smooth maps.
So if $G$ is a Lie group, stability is defined as above where the
family $\mathcal{G}$ consists of Lie groups and the set $\mathcal{F}$
consists only of smooth maps $G\to H_{i}$. Strict\textbackslash flexible
$F$-stability will therefore be defined for Lie groups only for $F=\RR$
or $\CC$, so that $\mathcal{G}=\mathcal{GL}\left(F\right)$ consists
of Lie groups.\medskip{}

Uniform stability of a Lie algebra $\mf g$ over a field $K$ can
be defined similarly to stability of groups, where $\mathcal{G}=\left\{ \mf h_{i}\right\} _{i\in I}$
is now a family of Lie algebras over a field $F$ extending $K$,
and $\mathcal{F}$ are all $F$-linear maps $\phi:\mf g\to\mf h_{i}$.
The defect of $\phi\in\mathcal{F}$ is then
\[
def\left(\phi\right)=sup\,\left\{ rk\left(\left[\phi\left(x\right),\phi\left(y\right)\right]-\phi\left(\left[x,y\right]\right)\right)\st x,y\in\mf g\right\} .
\]
Strict/flexible uniform $F$- stability of $\mf g$ is then taken
with respect to $\mathcal{G}=\left\{ \mf{gl}_{n}\left(F\right)\right\} _{n\in\NN}$
and $d^{strict},d^{flex}$ are defined similarly to the case of groups. 

\medskip{}

Stability of finitely presented algebras was considered in \cite{BauerBlacharGreenfeld},
including Lie algebras, with respect to pointwise stability. For finite
dimensional Lie algebras, pointwise stability is equivalent to uniform
stability; We can always fix a basis $\left(z_{1},\ldots,z_{m}\right)$
of $\mf g$ with structure constants $\left[z_{i},z_{j}\right]=\sum_{k}\alpha_{ij}^{k}z_{k}$
and get a presentation of $\mf g$ with these relations. By linearity
of $\phi$, the term $\left[\phi\left(x\right),\phi\left(y\right)\right]-\phi\left(\left[x,y\right]\right)$
can be expanded into at most $m^{2}$ scalar multiples of the terms
$\left[\phi\left(z_{i}\right),\phi\left(z_{j}\right)\right]-\phi\left(\sum_{k}\alpha_{ij}^{k}z_{k}\right)$,
which are the defining relations of $\mf g$. So the uniform defect
is at most $m^{2}$ times the pointwise defect $\underset{i,j}{sup}\,\left\{ rk\left(\left[\phi\left(z_{i}\right),\phi\left(z_{j}\right)\right]-\sum_{k}\alpha_{ij}^{k}\phi\left(z_{k}\right)\right)\right\} $,
and the distance $\underset{x\in\mf g}{sup}\;rk\left(\phi\left(x\right)-\theta\left(x\right)\right)$
again by linearity, triangle inequality and invariance to scalar multiples,
is at most $m$ times the pointwise distance $\underset{i}{sup}\,\left\{ rk\left(\phi\left(z_{i}\right)-\theta\left(z_{i}\right)\right)\right\} $.

It is worth noting that one can similarly show equivalence between
pointwise and uniform stability of finite dimensional \emph{associative}
algebras, although we do not discuss associative algebras in this
paper.

Like in the case of groups, flexible stability of finite dimensional
Lie algebras can also be reduced to a small increase of the dimension;
By Lemma 3.2 in \cite{BauerBlacharGreenfeld}, if $\mf g$ is a Lie
algebra generated by $d$ elements, $\ph:\mf g\to\mf{gl}_{n}\left(F\right)$
is a $\delta$-representation and $\psi:\mf g\to\mf{gl}_{N}\left(F\right)$
a representation such that $n\le N$ and $d^{flex}\left(\ph,\psi\right)\le\varepsilon$
then we may assume that $N\le\left(1+\varepsilon d\right)n$.

\subsection{Compressions and strict vs. flexible stability}

Fix a subspace $U\le F^{n}$ of dimension $k$, a basis $B=\left\{ u_{1},\ldots,u_{k}\right\} $
for $U$, $\iota:U\to F^{n}$ the inclusion and $p:F^{n}\twoheadrightarrow U$
a projection, i.e. a linear map satisfying $p^{2}=p$ and $im\,p=U$.
For any $M\in M_{n}\left(F\right)$, we denote by 
\[
\left[M\right]_{U,B,p}\vcentcolon=p\circ M\circ\iota\in M_{k}\left(F\right)
\]
the matrix representing the linear map $p\circ M\circ\iota\in End\left(U\right)$
in the basis $B$, and call it the \textbf{compression of $M$ associated
to $\left(U,B,p\right)$.}\emph{ }We will often omit $B$ and $p$
from the notation writing $\left[M\right]_{U}$, where a choice of
a basis and a projection are implicit.
\begin{lem}
\label{lem:compression} 
\begin{enumerate}
\item \label{enu:compression1}Let $U\le F^{n}$ be a subspace of dimension
$k$, and $\left[M\right]_{U,B,p}$ the compression associated to
$U$ with a basis $B$ and a projection $p$. Then we have
\[
rank\left(\left[M\right]_{U,B,p}\right)\ge rank\left(M\right)-2\left(n-k\right)\quad\forall M\in M_{n}\left(F\right),
\]
and for any $M_{1},M_{2}\in M_{n}\left(F\right)$
\[
rank\left(\left[M_{1}M_{2}\right]_{U,B,p}-\left[M_{1}\right]_{U,B,p}\left[M_{2}\right]_{U,B,p}\right)\le n-k.
\]
\item \label{enu:compression2}Let $U_{1},U_{2}\le F^{n}$ be subspaces
of dimension $k$, $\iota_{i}:U_{i}\inj F^{n}$, $p_{i}:F^{n}\twoheadrightarrow U_{i}$
inclusions and projections, and $B_{i}$ bases for $U_{i}$, respectively.
Then for some $A\in GL_{k}\left(F\right)$,
\[
rank\left(\left[M\right]_{U_{1},B_{1},p_{1}}-A\left[M\right]_{U_{2},B_{2},p_{2}}A^{-1}\right)\le4\left(n-k\right),\quad\forall M\in M_{n}\left(F\right).
\]
\end{enumerate}
\end{lem}

\begin{proof}
We have by the triangle inequality
\begin{align*}
rank\left(\left[M\right]_{U}\right) & =rank\left(p\circ M\circ\iota\right)=rank\left(\iota\circ p\circ M\circ\iota\circ p\right)\\
 & =rank\left(M-M\left(I_{n}-\iota\circ p\right)-\left(I_{n}-\iota\circ p\right)M\circ\iota\circ p\right)\\
 & \ge rank\left(M\right)-2rank\left(I_{n}-\iota\circ p\right)=rank\left(M\right)-2\left(n-k\right).
\end{align*}
For $M_{1},M_{2}\in M_{n}\left(F\right)$,
\begin{align*}
rank\left(\left[M_{1}M_{2}\right]_{U}-\left[M_{2}\right]_{U}\left[M_{2}\right]_{U}\right) & =rank\left(p\circ\left(M_{1}M_{2}-M_{1}\left(\iota\circ p\right)M_{2}\right)\circ\iota\right)\\
 & \le rank\left(M_{1}M_{2}-M_{1}\left(\iota\circ p\right)M_{2}\right)\\
 & =rank\left(M_{1}\left(I_{n}-\iota\circ p\right)M_{2}\right)\\
 & \le rank\left(I_{n}-\iota\circ p\right)=n-k,
\end{align*}
proving \ref{enu:compression1}. For \ref{enu:compression2}, let
$W=U_{1}\cap U_{2}$, and let $B_{W}=\left\{ w_{1},\ldots,w_{m}\right\} $
be a basis for $W$. By conjugating the whole expression 
\[
\left[M\right]_{U_{1},B_{1},p_{1}}-A\left[M\right]_{U_{2},B_{2},p_{2}}A^{-1}
\]
by some matrix in $GL_{k}\left(F\right)$, we may assume that the
basis $B_{1}$ is some completion of $B_{W}$ to a basis of $U_{1}$.
Now by choosing $A\in GL_{k}\left(F\right)$ we may further assume
that the basis $B_{2}$ is given by some completion of $B_{W}$ to
a basis of $U_{2}$ and that the expression is given by 
\[
\left[M\right]_{U_{1},B_{1},p_{1}}-\left[M\right]_{U_{2},B_{2},p_{2}}.
\]
Now for $w_{i}\in W$, 
\[
\left(p_{1}\circ M\circ\iota_{1}-p_{2}\circ M\circ\iota_{2}\right)w_{i}=\left(p_{1}-p_{2}\right)Mw_{i},
\]
thus the restriction of $\left[M\right]_{U_{1},B_{1},p_{1}}-\left[M\right]_{U_{2},B_{2},p_{2}}$
to the subspace $W$ that has dimension at least $2k-n$ has rank
at most $rank\left(p_{1}-p_{2}\right)\le n-dimW\le2\left(n-k\right)$.
It follows that 
\[
rank\left(\left[M\right]_{U_{1},B_{1},p_{1}}-\left[M\right]_{U_{2},B_{2},p_{2}}\right)\le2\left(n-k\right)+n-\left(2k-n\right)=4\left(n-k\right).
\]
\end{proof}
Lemma \ref{lem:compression} implies that a compression of an almost
representation (and in particular, of a representation) onto a large
subspace is again an almost representation (and that compressions
of nonsingular matrices have almost full rank), and that up to low
rank perturbation and conjugation, all $k$-dimensional compressions
of almost representations are equivalent. In particular, if all representations
have enough subrepresentations, this implies that flexible stability
is equivalent to strict stability.
\begin{cor}
\label{cor:strictflex}Let $F$ be an algebraically closed field of
characteristic zero. Then a solvable Lie algebra $\mf g$ over $K\le F$
is flexibly uniformly $F$-stable if and only if it is strictly uniformly
$F$-stable.
\end{cor}

\begin{proof}
Strict stability clearly implies flexible stability. Assume $\mf g$
is flexibly stable. Then for any $\delta>0$ there is an $\varepsilon>0$
such that for any $\delta$-representation $\ph:\mf g\to\mf{gl}_{n}\left(F\right)$
there is a representation $\psi:\mf g\to\mf{gl}_{N}\left(F\right)$
such that 
\[
\underset{x\in\mf g}{sup}\,rk\left(\widehat{\ph\left(x\right)}-\widehat{\psi\left(x\right)}\right)\le\varepsilon,
\]
and we may assume $N\le\left(1+\varepsilon d\right)n$, where $d$
is the number of generators of $\mf g$. By Lie's theorem, there is
a complete flag $F^{N}=V_{N}\supset V_{N-1}\supset\cdots\supset V_{0}=0$
of invariant subspaces of $\psi$ with $dim\,V_{i}=i$. Then for any
$k$ (and choice of projection and basis $p_{k},B_{k}$), the compression
\[
\left[\psi\right]_{V_{k}}=p_{V_{k}}\circ\psi\circ\iota_{V_{k}}:\mf g\to\mf{gl}_{n}\left(F\right)
\]
is a representation, and 
\[
\underset{x\in\mf g}{sup}\;rk\left(p_{V_{k}}\circ\widehat{\ph\left(x\right)}\circ\iota_{V_{k}}-\widehat{\left[\psi\right]_{V_{k}}\left(x\right)}\right)\le\underset{x\in\mf g}{sup}\;rk\left(\widehat{\ph\left(x\right)}-\widehat{\psi\left(x\right)}\right)\le\varepsilon.
\]
Thinking of the matrices as living in $\mf{gl}_{N}\left(F\right)$,
by definition of $\widehat{\ph\left(x\right)}$ it has an $n$-dimensional
compression that equals $\ph$, so by Lemma \ref{lem:compression}(\labelcref*{enu:compression2})
there is a $A\in GL_{n}\left(F\right)$ such that 
\[
rank\left(\ph\left(x\right)-A\left(p_{V_{n}}\circ\widehat{\ph\left(x\right)}\circ\iota_{V_{n}}\right)A^{-1}\right)\le4\left(N-n\right)\le4\varepsilon dn\quad\forall x\in\mf g
\]
\begin{align*}
 & \implies\underset{x\in\mf g}{sup}\;rk\left(\ph\left(x\right)-A\left[\psi\right]_{V_{n}}A^{-1}\left(x\right)\right)\\
 & \le\underset{x\in\mf g}{sup}\;rk\left(\ph\left(x\right)-A\left(p_{V_{n}}\circ\widehat{\ph\left(x\right)}\circ\iota_{V_{n}}\right)A^{-1}\right)+rk\left(A\left(p_{V_{n}}\circ\widehat{\ph\left(x\right)}\circ\iota_{V_{n}}-\widehat{\left[\psi\right]_{V_{n}}\left(x\right)}\right)A^{-1}\right)\\
 & \le\left(4d+1\right)\varepsilon,
\end{align*}
where $A\left[\psi\right]_{V_{n}}A^{-1}$ is an $n$-dimensional representation.
\end{proof}
\begin{rem}
It is interesting to note a similarity between Corollary \ref{cor:strictflex}
and strict vs. flexible pointwise stability of groups. In \cite[Proposition 6.2]{MR4579918}
it was proven that for amenable groups (and in particular solvable
groups), strict stability is equivalent to flexible stability in the
Hilbert-Schmidt norm. A similar result holds for permutation stability
as proved in \cite[Lemma 3.2]{MR4134896}.
\end{rem}

In the case of finite dimensional $\mf g$, Corollary \ref{cor:strictflex}
can be helpful only in the abelian case, as we know from \cite{BauerBlacharGreenfeld}
that nonabelian solvable Lie algebras are not flexibly stable. We
note that stability of abelian Lie algebras are equivalent to \emph{pointwise}
stability of free abelian groups, since from bi-invariance
\[
rk\left(ABA^{-1}B^{-1}-I_{n}\right)=rk\left(AB-BA\right),
\]
so the same corollary holds for free abelian groups. The proof of
Corollary \ref{cor:strictflex} is easily seen to hold also for \emph{uniform}
stability of free abelian groups. However, currently it is not even
known whether $\ZZ$ is uniformly $F$-stable for any field $F$.
By contrast, finite groups were proven to be $F$-stable in \cite[Section 6]{MallahiYekta},
for any field $F$. Pointwise and uniform stability for finite groups
are easily seen to be equivalent, similar to the case of finite dimensional
Lie algebras.
\begin{rem}
We have defined uniform stability only for individual groups/algebras.
However, it is also common to consider uniform stability for families
of groups, such as amenable groups in Kazhdan's theorem, where the
dependence of $\delta$ on $\varepsilon$ is uniform for the entire
family. The uniform rank metric stability of finite groups in \cite[Section 6]{MallahiYekta}
depends on the size of the group, and so it does not prove uniform
stability for the \emph{class} of finite groups.
\end{rem}

\subsection{Verma modules and central characters}

Let $\mf g$ be a semisimple Lie algebra over $\CC$, with a Cartan
subalgebra $\mf h$ of dimension $\ell$, the rank of $\mf g$. Denote
by $\Phi\subset\mf h^{*}$ the root system of $\mf g$ relative to
$\mf h$. To each root $\alpha$ denote by $\mf g_{\alpha}$ the corresponding
root space, and fix a set of positive roots $\Phi^{+}\subset\Phi$
and a set of simple roots $\Delta\subset\Phi^{+}$ having $\ell$
elements. This defines a Cartan decomposition
\[
\mf g=\mf n^{-}\oplus\mf h\oplus\mf n^{+},
\]
where 
\[
\mf n^{+}\vcentcolon=\bigoplus_{\alpha>0}\mf g_{\alpha}\;,\;\mf n^{-}\vcentcolon=\bigoplus_{\alpha<0}\mf g_{\alpha},
\]
with corresponding Borel subalgebras 
\[
\mf b^{+}=\mf h\oplus\mf n^{+}\;,\;\mf b^{-}=\mf h\oplus\mf n^{-}
\]

Let $\Lambda_{r}$ be the root lattice and $\Lambda$ be the integral
weight lattice in $\Phi$, and $\Lambda^{+}\sq\Lambda$ the dominant
integral weights. Let $\rho$ be the half sum of the positive roots,
also called the Weyl vector. Denote by $W$ the (finite) Weyl group
of $\Phi$, generated by the reflections $s_{\alpha}$. 

There is a standard basis for $\mf g$ consisting of $y_{1},\ldots,y_{m}\in\mf n^{-}$,
$h_{1},\ldots,h_{\ell}\in\mf h$ and $x_{1}\ldots,x_{m}\in\mf n^{+}$,
such that $x_{i}\in\mf g_{\alpha_{i}}$ for positive roots $\alpha_{i}\in\Phi^{+}$,
$y_{i}\in\mf g_{-\alpha_{i}}$, and $h_{1}=h_{\alpha_{i}}$ for simple
roots $\alpha_{i}\in\Delta$. According to the Poincaré-Birkhoff-Witt
(PBW) theorem, there is a corresponding basis for the universal enveloping
algebra $U\left(\mf g\right)$ consisting of monomials 
\[
y_{1}^{r_{1}}\cdots y_{m}^{r_{m}}h_{1}^{s_{1}}\cdots h_{\ell}^{s_{\ell}}x_{1}^{t_{1}}\cdots x_{m}^{t_{m}}
\]

A $U\left(\mf g\right)$-module is a highest weight module of weight
$\lambda\in\mf h^{*}$ if there is a vector $v_{0}\in M$ of weight
$\lambda$ such that $\mf n^{+}\cdot v_{0}=0$ and $M=U\left(\mf g\right)\cdot v_{0}$.
By the PBW theorem, such a module satisfies $M=U\left(\mf n^{-}\right)\cdot v_{0}$.
A universal highest weight module of weight $\lambda$ is called a
Verma module and is denoted by $M\left(\lambda\right)$. One may construct
it by induction from $\mf b$: Since $\mf b/\mf n$ is isomorphic
to $\mf h$, there is a $1$-dimensional representation $\CC_{\lambda}$
of $\mf b$ with trivial $\mf n$-action and $\mf h$ acting as $\lambda$.
By PBW, $U\left(\mf b\right)\le U\left(\mf g\right)$ is a subalgebra
and thus acts on $U\left(\mf g\right)$ by left multiplication. 

Set 
\[
M\left(\lambda\right)\vcentcolon=Ind_{\mf b}^{\mf g}\CC_{\lambda}=U\left(\mf g\right)\otimes_{U\left(\mf b\right)}\CC_{\lambda}
\]
Which has a structure of a left $U\left(\mf g\right)$-module with
highest weight vector $1\otimes1$. Again by the PBW theorem, $U\left(\mf g\right)\cong U\left(\mf n^{-}\right)\otimes U\left(\mf b\right)$
and thus $M\left(\lambda\right)\cong U\left(\mf n^{-}\right)\otimes\CC_{\lambda}$
as a left $U\left(\mf n^{-}\right)$-module, and is therefore a free
$U\left(\mf n^{-}\right)$-module of rank $1$. Every highest weight
module for $\lambda$ is a quotient of $M\left(\lambda\right)$, and
each $M\left(\lambda\right)$ has a unique simple quotient $L\left(\lambda\right)$
which is finite-dimensional if and only if $\lambda\in\Lambda^{+}$. 

We define the shifted action of $W$ on $\mf h^{*}$ by $w\cdot\lambda=w\left(\lambda+\rho\right)-\rho$,
and we say that $\lambda,\mu\in\mf h^{*}$ are \emph{$W$-linked }if
$\mu=w\cdot\lambda$ for some $w\in W$.

The center $Z\left(U\left(\mf g\right)\right)$ acts on a highest
weight module $M$ of weight $\lambda$ via scalars:
\[
z\cdot v=\chi_{\lambda}\left(z\right)v\quad\forall v\in M
\]
Thus $\chi_{\lambda}:Z\left(U\left(\mf g\right)\right)\to\CC$ is
an algebra homomorphism that we call the \emph{central character}
associated with $\lambda$. A fundemental theorem on central characters
is due to Harish-Chandra \cite[Theorem 1.10]{HumphreysBGG}:
\begin{thm*}
(Harish-Chandra isomorphism)
\begin{enumerate}
\item There is an isomorphism of algebras $Z\left(U\left(\mf g\right)\right)\cong S\left(\mf h\right)^{W}$
where $S\left(\mf h\right)^{W}$ is the algebra of $W$-invariant
elements of the symmetric algebra of $\mf h$.
\item For $\lambda,\mu\in\mf h^{*}$, $\chi_{\lambda}=\chi_{\mu}$ if and
only if $\mu$ and $\lambda$ are $W$-linked.
\item Every central character $\chi:Z\left(U\left(\mf g\right)\right)\to\CC$
is of the form $\chi_{\lambda}$ for some $\lambda\in\mf h^{*}$.
\end{enumerate}
The algebra $S\left(\mf h\right)^{W}$ of invariants is isomorphic
to a polynomial algebra in $\ell>0$ variables, and therefore $Z\left(U\left(\mf g\right)\right)$
is (freely) generated as an algebra by $\ell$ elements.
\end{thm*}

\subsection{Lattices in semisimple Lie groups}

Let $G$ be a semisimple Lie group. A discrete subgroup $\Gamma\le G$
is a \emph{lattice} if $G/\Gamma$ has finite volume (i.e. there is
a $\sigma$-finite, $G$-invariant Borel measure $\nu$ on $G/\Gamma$).
Lattices possess various rigidity properties, one of the most fundemental
properties is the Borel Density Theorem. The following formulation
is taken from \cite[Corollary 4.5.6]{WitteMorrisArithmetic}.
\begin{thm*}[Borel Density]
Let $G\le SL_{r}\left(\RR\right)$ be a connected linear semisimple
Lie group and $\Gamma\le G$ a lattice. If $\Gamma$ projects densely
onto the maximal compact factor of $G$, then $\Gamma$ is Zariski
dense in $G$. That is, if $Q\in\RR\left[x_{1,1},\ldots,x_{r,r}\right]$
is a polynomial function on $M_{r}\left(\RR\right)$ such that $Q\left(\Gamma\right)=0$,
then $Q\left(G\right)=0$.
\end{thm*}
A remarkable theorem on lattices in higher rank is Margulis Superrigidity,
which states that in higher rank $G$ and $\Gamma\le G$ an irreducible
lattice, essentially all finite dimensional representations of $\Gamma$
can be extended to rational representations of $G$ up to some bounded
error. For noncocompact $\Gamma$, any representation can be extended
to a rational representation of $G$ on some finite index subgroup.
For cocompact lattices this is not necessarily true - there may be
unitary representations that do not factor through a finite quotient.
But we know from Margulis arithmeticity that $\Gamma$ is an arithmetic
subgroup of $G$, and therefore we may add compact factors to $G$
so that $\Gamma$ is commensurable to the integral points of $G'=G\times K$
(these compact factors are in fact Galois twistings corresponding
to the number field defining the arithmetic group $\Gamma$, see \cite[Section 5.5]{WitteMorrisArithmetic}).
Then it is true that finite dimensional representations of $\Gamma$
extend on some finite index subgroup to a representation of $G'$
(see \cite[Section 16.4]{WitteMorrisArithmetic}).
\begin{thm*}[Margulis Superrigidity]
Let $G$ be a connected, algebraically simply connected semisimple
Lie group of rank $\ge2$ without compact factors that is defined
over $\mathbb{Q}$ and let $\Gamma\le G$ be an irreducible lattice.
Then there is a compact semisimple group $K$, $G'=G\times K$ where
$\Gamma\le G'$ is an irreducible lattice projecting densely onto
$K$, such that for any representation $\ph:\Gamma\to GL_{n}\left(\CC\right)$
there is a finite index subgroup $\tilde{\Gamma}\le\Gamma$ and a
rational representation $\Phi:G'\to GL_{n}\left(\CC\right)$ such
that $\Phi\mid_{\tilde{\Gamma}}=\ph$.
\end{thm*}

\section{Semisimple Lie algebras and Lie groups}

\subsection{Non-stability of semisimple Lie algebras}

For any finitely generated Lie algebra $\mf g$ and $\varepsilon$-representation
$\ph:\mf g\to\mf{gl}_{n}\left(\CC\right)$, we may extend $\ph$ to
all of $U\left(\mf g\right)$ and this will be a \emph{pointwise}
asymptotic representation of the associative algebra $U\left(\mf g\right)$.
This is because pointwise asymptotic representations are determined
by the action of the generators, and the generators of $\mf g$ as
a Lie algebra generate $U\left(\mf g\right)$ as an associative algebra.

Let $\mf g$ be a finite dimensional complex semisimple Lie algebra.
The following theorem shows that although $\mf g$ has only finitely
many $n$-dimensional representations, it has uncountably many $\varepsilon_{n}$-representations
(for some $\varepsilon_{n}\conv{n\to\infty}0$) that are nonequivalent:
\begin{thm}
\label{thm:thm9}Let $\mf g$ be a complex semisimple Lie algebra
of rank $\ell$ and $m=\left|\Phi^{+}\right|$ the number of positive
roots. Then there is a family 
\[
\left\{ \ph_{n}^{\lambda}:\mf g\to\mf{gl}_{k_{n}}\left(\CC\right)\st\lambda\in\mf h^{*}\right\} 
\]
of $\varepsilon_{n}$-representations, where $\varepsilon_{n}=\frac{2m^{2}}{n}$
and $k_{n}=\binom{n+m}{m}-1$, satisfying the following properties:
\begin{enumerate}
\item \label{enu:10.1}The vector space $\CC^{k_{n}}=V=\bigoplus_{\mu\in\mf h^{*}}V_{\mu}$
is a sum of weight spaces, and there is a vector $v_{0}$ of weight
$\lambda$ such that $\ph_{n}^{\lambda}\left(\mf n^{+}\right)\cdot v_{0}=0$
and $\ph_{n}^{\lambda}\left(\mf n^{-}\right)=V$. 
\item \label{enu:10.2}Let $z_{1},\ldots,z_{\ell}$ be generators of $Z\left(U\left(\mf g\right)\right)$
as an algebra and $R=\underset{i}{max}\,deg\left(z_{i}\right)$. Denote
by $\tilde{\ph}_{n}^{\lambda}$ its extension to $U\left(\mf g\right)$.
Then
\[
rk\left(\tilde{\ph}_{n}^{\lambda}\left(z_{i}\right)-\chi_{\lambda}\left(z_{i}\right)I_{k_{n}}\right)\le\left(R+1\right)\varepsilon_{n}
\]
\item \label{enu:10.3}For any $\lambda,\mu\in\mf h^{*}$that are not $W$-linked,
\begin{align*}
d^{strict}\left(\ph_{n}^{\lambda},A\ph_{n}^{\mu}A^{-1}\right) & \ge\left(\binom{R+\ell}{\ell}R\right)^{-1}\left(1-2\left(R+1\right)\varepsilon_{n}\right)\\
 & \conv{n\to\infty}\left(\binom{R+\ell}{\ell}R\right)^{-1}
\end{align*}
For any $A\in GL_{k_{n}}\left(\CC\right)$.
\item \label{enu:10.4}For any $\lambda\in\mf h^{*}$ and any representation
$\psi$ of $\mf g$ such that $\lambda$ is not $W$-linked to any
$\mu$ appearing as a highest weight of a subrepresentation of $\psi$,
\begin{align*}
d^{flex}\left(\ph_{n}^{\lambda},\psi\right) & \ge\left(\binom{R+\ell}{\ell}R\ell\right)^{-1}\left(1-\ell\left(R+m+1\right)\varepsilon_{n}\right)\\
 & \conv{n\to\infty}\left(\binom{R+\ell}{\ell}R\ell\right)^{-1}
\end{align*}
In particular, if $\lambda\in\mf h^{*}$ is not $W$-linked to any
$\mu\in\Lambda^{+}$, then $\ph_{n}^{\lambda}$ is bounded from representations
of $\mf g$, and thus $\mf g$ is not flexibly $\CC$-stable.
\end{enumerate}
\end{thm}

The maps $\ph_{n}^{\lambda}$ are constructed via finite dimensional
compressions of infinite dimensional weight representations of $\mf g$
onto subspaces that are ``almost invariant''. Looking at weight
diagrams of such infinite dimensional representations, it is apparent
that $\mf g$ acts ``locally'' on the weight spaces, so that if
one takes a large ball in $\mf{h^{*}}$, the relations of $\mf g$
are satisfied on the interior and not on the boundary, which can be
shown to be small with respect to the whole subspace. One can instead
note that $U\left(\mf g\right)$ has polynomial growth and use this
to construct the almost invariant subspaces taking polynomials up
to a given degree in quotients of the $\mf g$-module $U\left(\mf g\right)$.
\begin{proof}
We construct $\ph_{n}^{\lambda}$ by compressing the Verma module
$M\left(\lambda\right)$ to a finite dimensional subspace. Recall
that $M\left(\lambda\right)$ is free as a $U\left(\mf n^{-}\right)$-module,
and therefore has a basis of the form 
\[
M\left(\lambda\right)=span\left\{ y_{1}^{r_{1}}\cdots y_{m}^{r_{m}}\cdot v_{0}\st r_{i}\in\ZZ_{\ge0}\right\} ,
\]
where $y_{1},\ldots,y_{m}$ are the negative root spaces and $v_{0}$
the highest weight vector. For $n\in\NN$, Denote by 
\[
D_{n}=span\left\{ y_{1}^{r_{1}}\cdots y_{m}^{r_{m}}\cdot v_{0}\st\sum_{i=1}^{m}r_{i}\le n\right\} 
\]
the subspace spanned by all monomials in $U\left(\mf n^{-}\right)$
of degree $\le n$ acting on $v_{0}$. A standard partition argument
shows that there are $\binom{n+m-1}{m-1}$ such combinations for fixed
$a$, and
\[
dim\left(D_{n}\right)=\sum_{j=1}^{n}\binom{j+m-1}{m-1}=\binom{n+m}{m}-1.
\]

Define the action of $\mf g$ on $D_{n}$ by
\[
\left(\ph_{n}^{\lambda}\left(x\right)\right)\left(y_{1}^{r_{1}}\cdots y_{m}^{r_{m}}\cdot v_{0}\right)=\begin{cases}
xy_{1}^{r_{1}}\cdots y_{m}^{r_{m}}\cdot v_{0} & xy_{1}^{r_{1}}\cdots y_{m}^{r_{m}}\cdot v_{0}\in D_{n}\\
0 & else
\end{cases}.
\]
We claim that the image of this action of $\mf g$ on $D_{n-1}$ is
contained in $D_{n}$; For $\mf n^{-}$, after reordering using a
PBW basis for $\mf n^{-}$ we have a sum of monomials of degree at
most $n$. The action of $\mf h$ is by scalars and therefore any
$D_{k}$ is $\mf h$-invariant. We prove by induction on $n$ that
$\mf n^{+}\cdot D_{n}\sq D_{n}$: When $n=0$ this is obvious. Assume
it holds for $n-1$, then for $x\in\mf g_{\beta}$ ($\beta>0$) we
have 
\[
xy_{1}^{r_{1}}y_{2}^{r_{2}}\cdots y_{m}^{r_{m}}\cdot v_{0}=y_{1}xy_{1}^{r_{1}-1}y_{2}^{r_{2}}\cdots y_{m}^{r_{m}}\cdot v_{0}+zy_{1}^{r_{1}-1}y_{2}^{r_{2}}\cdots y_{m}^{r_{m}}\cdot v_{0}
\]
for some $z\in\mf g_{-\alpha_{1}+\beta}$. By the induction hypothesis,
$xy_{1}^{r_{1}-1}y_{2}^{r_{2}}\cdots y_{m}^{r_{m}}\cdot v_{0}\in D_{n-1}\implies y_{1}xy_{1}^{r_{1}-1}y_{2}^{r_{2}}\cdots y_{m}^{r_{m}}\cdot v_{0}\in D_{n}$.
Moreover, if $\beta-\alpha_{1}>0$ then again use the induction hypothesis
for $zy_{1}^{r_{1}-1}y_{2}^{r_{2}}\cdots y_{m}^{r_{m}}\cdot v_{0}$,
and if $\beta-\alpha_{1}\le0$ then $z\in\mf n^{-}$ or $z\in\mf h$
and in any case $zy_{1}^{r_{1}-1}y_{2}^{r_{2}}\cdots y_{m}^{r_{m}}\cdot v_{0}\in D_{n}$. 

Now since $\ph_{n}^{\lambda}\left(\mf g\right)\cdot D_{n-1}\sq D_{n}$,
we have $\left[\ph_{n}^{\lambda}\left(x\right),\ph_{n}^{\lambda}\left(y\right)\right]\cdot D_{n-2}\sq D_{n}$
for any $x,y\in\mf g$. Since $\ph_{n}^{\lambda}$ acts as $M\left(\lambda\right)$
when restriced to $D_{n-2}$, we have $\left(\ph_{n}^{\lambda}\left(\left[x,y\right]\right)-\left[\ph_{n}^{\lambda}\left(x\right),\ph_{n}^{\lambda}\left(y\right)\right]\right)\cdot D_{n-2}=0$.
Thus 
\begin{align*}
rank\left(\ph_{n}^{\lambda}\left(\left[x,y\right]\right)-\left[\ph_{n}^{\lambda}\left(x\right),\ph_{n}^{\lambda}\left(y\right)\right]\right) & \le dim\,D_{n}-dim\,D_{n-2}=\binom{n+m-1}{n}+\binom{n+m-2}{n}\\
\implies\varepsilon_{n}\le\frac{\binom{n+m-1}{n}+\binom{n+m-2}{m-1}}{\binom{n+m}{m}} & =\frac{m}{n+m}\left(1+\frac{mn}{\left(n+m-1\right)}\right)\le\frac{m}{n+m}\left(1+m\right)<\frac{2m^{2}}{n}.
\end{align*}
All of the generators $z_{1},\ldots,z_{\ell}$ of $Z\left(U\left(\mf g\right)\right)$
have degree at most $R$, so when restricted to $D_{n-R}$ they act
as the scalars $\chi_{\lambda}\left(z_{i}\right)$. Thus
\begin{align*}
rank\left(\tilde{\ph}_{n}^{\lambda}\left(z_{i}\right)-\chi_{\lambda}\left(z_{i}\right)I_{k_{n}}\right) & \le dimD_{n}-dimD_{n-R}\\
 & =\binom{n+m}{m}-\binom{n-R+m}{m}
\end{align*}
\begin{align*}
 & \implies rk\left(\tilde{\ph}_{n}^{\lambda}\left(z_{i}\right)-\chi_{\lambda}\left(z_{i}\right)I_{k_{n}}\right)\le1-\frac{\binom{n-R+m}{m}}{\binom{n+m}{m}}\le1-\frac{\left(n-R\right)^{m}}{\left(n+m\right)^{m}}.
\end{align*}
Using the inequality $log\left(1+x\right)\ge\frac{x}{1+x}$, 
\[
log\left(\frac{\left(n-R\right)^{m}}{\left(n+m\right)^{m}}\right)=mlog\left(1-\frac{m+R}{n+m}\right)\ge-m\frac{m+R}{n-R},
\]
so that
\begin{align*}
rk\left(\tilde{\ph}_{n}^{\lambda}\left(z_{i}\right)-\chi_{\lambda}\left(z_{i}\right)I_{k_{n}}\right) & \le1-exp\left(-m\frac{m+R}{n-R}\right)\le m\frac{m+R}{n-R}\\
 & <m\frac{m+2R}{n}<\left(R+1\right)\varepsilon_{n}.
\end{align*}
This proves \ref{enu:10.2}. 

For \ref{enu:10.3}, by Harish-Chandra's theorem $\chi_{\lambda}\neq\chi_{\mu}$
so that they differ on at least one of the generators, say $z_{j}$.
Therefore, for any $A\in GL_{k_{n}}\left(\CC\right)$, 
\begin{align*}
rk\left(\tilde{\ph}_{n}^{\lambda}\left(z_{j}\right)-A\tilde{\ph}_{n}^{\mu}\left(z_{j}\right)A^{-1}\right) & \ge rk\left(\chi_{\lambda}\left(z_{j}\right)-A\chi_{\mu}\left(z_{j}\right)A^{-1}\right)-2\left(R+1\right)\varepsilon_{n}\\
 & =1-2\left(R+1\right)\varepsilon_{n}.
\end{align*}
Let $\eta=\ph_{n}^{\lambda}-A\ph_{n}^{\mu}A^{-1}$ and $\delta=\underset{x\in\mf g}{max}\;rk\left(\eta\left(x\right)\right)$.
Every $z_{i}$ is a polynomial of degree $\le R$ in the PBW basis,
so it is a sum of at most $\binom{R+\ell}{\ell}-1$ monomials of degree
at most $R$. By \cite[Lemma 3.3]{BauerBlacharGreenfeld} we therefore
have
\[
rk\left(\tilde{\ph}_{n}^{\lambda}\left(z_{j}\right)-A\tilde{\ph}_{n}^{\mu}\left(z_{j}\right)A^{-1}\right)\le\binom{R+\ell}{\ell}R\delta.
\]
Combining the above we get
\[
1-2\left(R+1\right)\varepsilon_{n}\le\binom{R+\ell}{\ell}R\delta
\]
\[
\implies\delta\ge\left(\binom{R+\ell}{\ell}R\right)^{-1}\left(1-2\left(R+1\right)\varepsilon_{n}\right),
\]
proving \ref{enu:10.3}.

Let $\psi:\mf g\to\mf{gl}_{N}\left(\CC\right)$ be a representation,
where $N\le\left(1+\varepsilon_{n}m\right)k_{n}$, and suppose that
$\lambda\in\mf h^{*}$is not $W$-linked to any $\mu\in\mf h^{*}$
appearing as a highest weight of a subrepresentation of $\psi$. By
complete reducibility, it is a (conjugation of) a direct sum of isotypic
components
\[
\psi=A\left(\bigoplus_{i=1}^{q}\psi_{\mu_{i}}\right)A^{-1},
\]
where $\psi_{\mu_{i}}$ is the direct sum of copies of the irreducible
representation with highest weight $\mu_{i}$ (which must in particular
be dominant integral weights). The central character of $\psi$ is
therefore
\[
\chi_{\psi}=A\left(\bigoplus_{i=1}^{q}\chi_{\mu_{i}}I_{d_{i}}\right)A^{-1}.
\]

Let $\eta=\widehat{\ph_{n}^{\lambda}}-\widehat{\psi}$ and $\delta=\underset{x\in\mf g}{max}\;rk\left(\eta\left(x\right)\right)$.
Similar to the calculation above, for each $1\le i\le\ell$ we have
\begin{align*}
rk\left(\tilde{\psi}\left(z_{i}\right)-\chi_{\lambda}\left(z_{i}\right)I_{k_{n}}\right) & \le rk\left(\tilde{\psi}\left(z_{i}\right)-A\tilde{\ph}_{n}^{\lambda}\left(z_{j}\right)A^{-1}\right)+rk\left(A\tilde{\ph}_{n}^{\lambda}\left(z_{j}\right)A^{-1}-\chi_{\lambda}\left(z_{i}\right)I_{k_{n}}\right)\\
 & \le\binom{R+\ell}{\ell}R\delta+\left(R+1\right)\varepsilon_{n}.
\end{align*}
Since $N-k_{n}\le\varepsilon_{n}m$, we have 
\begin{align*}
rk\left(\tilde{\psi}\left(z_{i}\right)-\chi_{\lambda}\left(z_{i}\right)I_{N}\right) & \le\binom{R+\ell}{\ell}R\delta+\left(R+1\right)\varepsilon_{n}+\left(N-k_{n}\right)\\
 & \le\binom{R+\ell}{\ell}R\delta+\left(R+m+1\right)\varepsilon_{n}.
\end{align*}
Thinking of $\widehat{\ph_{n}^{\lambda}},\psi$ as embedded in $\mf{gl}_{N}\left(\CC\right)$,
the subspaces $W_{i}=ker\left(\tilde{\psi}\left(z_{i}\right)-\chi_{\lambda}\left(z_{i}\right)I_{N}\right)\le\CC^{N}$
have dimension at least $k_{n}\left(1-\binom{R+\ell}{\ell}R\delta-\left(R+m+1\right)\varepsilon_{n}\right)$
and for any $v\in W$,
\[
\tilde{\psi}\left(z_{i}\right)v=\chi_{\lambda}\left(z_{i}\right)v,
\]
hence $\tilde{\psi}\left(z_{i}\right)$ has an eigenspace of dimension
at least $k_{n}\left(1-\binom{R+\ell}{\ell}R\delta-\left(R+m+1\right)\varepsilon_{n}\right)$
corresponding to the eigenvalue $\chi_{\lambda}\left(z_{i}\right)$.
Suppose the intersection
\[
W=\bigcap_{i=1}^{\ell}W_{i}
\]
contains a nonzero vector $0\neq w\in W$, then $w$ satisfies for
any $1\le i\le\ell$
\[
\tilde{\psi}\left(z_{i}\right)w=\chi_{\lambda}\left(z_{i}\right)w,
\]
and in particular it follows that $w$ must be in some isotypic component
of $\psi$. This is a contradiction because $\chi_{\lambda}\neq\chi_{\mu_{i}}$
by assumption and Harish-Chandra's theorem, so $W=0$. But since all
$W_{i}$ are subspaces of $\CC^{N}$,
\begin{align*}
0=dim\,W & \ge N-\sum_{i=1}^{\ell}\left(N-dim\,W_{i}\right)\ge N-\ell k_{n}\left(\binom{R+\ell}{\ell}R\delta+\left(R+m+1\right)\varepsilon_{n}\right)\\
 & \ge N\left(1-\ell\left(\binom{R+\ell}{\ell}R\delta+\left(R+m+1\right)\varepsilon_{n}\right)\right)
\end{align*}
\begin{align*}
 & \implies1\le\ell\left(\binom{R+\ell}{\ell}R\delta\right)+\ell\left(R+m+1\right)\varepsilon_{n}\\
 & \delta\ge\left(\binom{R+\ell}{\ell}R\ell\right)^{-1}\left(1-\ell\left(R+m+1\right)\varepsilon_{n}\right)
\end{align*}
\end{proof}
By complexifying real Lie algebras, we can prove similar results for
real Lie algebras.
\begin{prop}
\label{prop:10}Let $\mf g$ be a finite dimensional real Lie algebra
and $\ph:\mf g\to\mf{gl}_{n}\left(\CC\right)$ an $\varepsilon$-representation. 

Then there is a $4\varepsilon$-representation $\tilde{\ph}:\mf g_{\CC}\to\mf{gl}_{n}\left(\CC\right)$
extending $\ph$. Moreover, if $d^{strict}\left(\ph,\psi\right)=\delta$
then $d^{strict}\left(\tilde{\ph},\tilde{\psi}\right)\le2\delta$
(and similarly for $d^{flex}$).
\end{prop}

\begin{proof}
Define $\tilde{\ph}$ in the obvious way by $\tilde{\ph}\left(x+iy\right)=\ph\left(x\right)+i\ph\left(y\right)$.
Then 
\[
\left[\tilde{\ph}\left(x+iy\right),\tilde{\ph}\left(w+iv\right)\right]=\left[\ph\left(x\right),\ph\left(w\right)\right]-\left[\ph\left(y\right),\ph\left(v\right)\right]+i\left(\left[\ph\left(x\right),\ph\left(v\right)\right]+\left[\ph\left(y\right),\ph\left(w\right)\right]\right)
\]
\[
\tilde{\ph}\left(\left[x+iy,w+iv\right]\right)=\ph\left(\left[x,w\right]\right)-\ph\left(\left[y,v\right]\right)+i\left(\ph\left(\left[x,v\right]\right)+\ph\left(\left[y,w\right]\right)\right),
\]
and it follows that $\tilde{\ph}$ is a $4\varepsilon$-representation. 

If $\psi$ is a representation of $\mf g$ such that $d^{strict}\left(\ph,\psi\right)=\delta$
(respectively $d^{flex}\left(\ph,\psi\right)=\delta$) then for any
$x+iy\in\mf g_{\CC}$,
\[
rk\left(\tilde{\ph}\left(x+iy\right)-\tilde{\psi}\left(x+iy\right)\right)\le rk\left(\ph\left(x\right)-\psi\left(x\right)\right)+rk\left(i\left(\ph\left(y\right)-\psi\left(y\right)\right)\right)\le2\delta
\]
\end{proof}
\begin{cor}
\label{cor:real}A real Lie algebra $\mf g$ is (flexibly) $\CC$-stable
if and only if $\mf g_{\CC}$ is (flexibly) $\CC$-stable (as a complex
Lie algebra). In particular, by Theorem \ref{thm:thm9}, real semisimple
Lie algebras are not flexibly $\CC$-stable.
\end{cor}

\begin{proof}
If $\mf g$ is (flexibly) stable, then any $\varepsilon$-representation
$\tilde{\ph}$ of $\mf g_{\CC}$ restricts to an $\varepsilon$-representation
$\ph$ of $\mf g$, which satisfies $d^{strict}\left(\ph,\psi\right)\le\delta$
(respectively $d^{flex}\left(\ph,\psi\right)\le\delta$) for some
representation $\psi$. Complexifying both, by Proposition \ref{prop:10}
$\tilde{\psi}$ is a representation of $\mf g_{\CC}$ satisfying $d^{strict}\left(\tilde{\psi},\tilde{\ph}\right)\le2\delta$
(respectively $d^{flex}\left(\ph,\psi\right)\le2\delta$).

Conversely, if $\mf g_{\CC}$ is (flexibly) stable, then an $\frac{\varepsilon}{4}$-representation
$\ph$ of $\mf g$ can be extended by Proposition \ref{prop:10} to
an $\varepsilon$-representation $\tilde{\ph}$ of $\mf g_{\CC}$.
Hence there is a representation $\tilde{\psi}$ of $\mf g_{\CC}$
satisfying $d^{strict}\left(\tilde{\ph},\tilde{\psi}\right)\le\delta$
(respectively $d^{flex}\left(\tilde{\ph},\tilde{\psi}\right)\le\delta$),
and restricting to $\mf g$ we have a representation $\psi$ of $\mf g$
with $d^{strict}\left(\ph,\psi\right)\le\delta$ (respectively $d^{flex}\left(\ph,\psi\right)\le\delta$).
\end{proof}
Notice that the proof of Theorem \ref{thm:thm9}(\ref{enu:10.4})
only used the central character of $\ph_{n}^{\lambda}$, which depends
solely on $\lambda$. If we consider a sequence of \emph{actual} representations
$\ph_{n}^{\lambda_{n}}:\mathfrak{g}\to M_{k_{n}}\left(\CC\right)$
with highest weight $\lambda_{n}$ (and corresponding central character),
then by an argument similar to the proof of Theorem \ref{thm:thm9}(\ref{enu:10.4})
we can conclude that it is bounded from any other representations
(up to conjugation). 
\begin{prop}
\label{prop:rigidity}Let $\psi_{1}:\mathfrak{g}\to\mf{gl}_{n}\left(\CC\right)$
be a representation with highest weight $\lambda$, and let $\psi_{2}:\mathfrak{g}\to\mf{gl}_{n}\left(\CC\right)$
be another representation such that $d^{strict}\left(\psi_{1},\psi_{2}\right)<\left(\binom{R+\ell}{\ell}R\ell\right)^{-1}.$

Then $\psi_{1}$ and $\psi_{2}$ are conjugate.
\end{prop}

\begin{proof}
Suppose $\psi_{1}$ and $\psi_{2}$ are not conjugate. Then $\psi_{2}$
must be a direct sum of irreducible subrepresentations, non of which
can have $\lambda$ as highest weight. Arguing similarly to the proof
of Theorem \ref{thm:thm9}(\ref{enu:10.4}), we therefore have 
\[
d^{strict}\left(\psi_{1},\psi_{2}\right)\ge\left(\binom{R+\ell}{\ell}R\ell\right)^{-1}.
\]
\end{proof}

\subsection{Semisimple Lie groups and lattices}

One of the motivations to consider stability of Lie algebras is that
almost representations arise naturally as differentials of (smooth)
almost representations of Lie groups. Unlike matrix norms, the rank
metric is in general compatible with differentiation:
\begin{lem}
\label{lem:lem12}Let $B_{\delta}\left(M\right)\sq M_{n}\left(\CC\right)$
be the $\delta$-ball around $M$ in $M_{n}\left(\CC\right)$, 
\[
B_{\delta}\left(M\right)=\left\{ A\in M_{n}\left(\CC\right)\st rk\left(A-M\right)\le\delta\right\} ,
\]
and let $\gamma:\left(-1,1\right)\to B_{\delta}\left(M\right)$ be
a smooth curve. Then $\frac{d\gamma}{dt}:\left(-1,1\right)\to B_{2\delta}\left(0\right)$,
so that $rk\left(\frac{d\gamma}{dt}\mid_{t_{0}}\right)\le2\delta$.
\end{lem}

\begin{proof}
Being contained in $B_{\delta}\left(M\right)$ is equivalent to the
vanishing of all the $\left(k+1\right)$-minors of $\gamma\left(t\right)-M$
for $k=\left\lfloor \delta n\right\rfloor $, and is therefore a closed
condition - thus $rank\left(\bullet\right)$ is a lower semicontinuous
function. We therefore have 
\begin{align*}
rank\left(\lim{h\to0}\,\frac{\gamma\left(t_{0}+h\right)-\gamma\left(t_{0}\right)}{h}\right) & \le\underset{h\to0}{liminf}\,rank\left(\frac{\gamma\left(t_{0}+h\right)-\gamma\left(t_{0}\right)}{h}\right)\\
 & \le\underset{h}{sup}\;rank\left(\gamma\left(t_{0}+h\right)-M\right)+rank\left(M-\gamma\left(t_{0}\right)\right)\le2\delta n
\end{align*}
\end{proof}
\begin{prop}
\label{prop:diff}
\end{prop}

\begin{itemize}
\item Let $G$ be a Lie group and $\Phi:G\to GL_{n}\left(\CC\right)$ a
smooth uniform $\varepsilon$-representation. Then $\ph\vcentcolon=d\Phi\mid_{id}:\mf g\to\mf{gl}_{n}\left(\CC\right)$
is an $8\varepsilon$-representation. 
\item If $\Psi:G\to GL_{N}\left(\CC\right)$ is another smooth map, $\psi=d\Psi$
and $d^{strict}\left(\Phi,\Psi\right)\le\delta$ (respectively $d^{flex}\left(\Phi,\Psi\right)\le\delta$),
then $d^{strict}\left(\ph,\psi\right)\le2\delta$ ($d^{flex}\left(\ph,\psi\right)\le2\delta$).
\end{itemize}
\begin{proof}
For any $g,h\in G$ we have by the triangle inequality
\[
rk\left(\Phi\left(ghg^{-1}\right)-\Phi\left(g\right)\Phi\left(h\right)\Phi\left(g\right)^{-1}\right)\le2\varepsilon.
\]
Written differently setting $\Psi_{g}\left(h\right)=ghg^{-1}$,
\[
rk\left(\Phi\left(\Psi_{g}\left(h\right)\right)-\Psi_{\Phi\left(g\right)}\left(\Phi\left(h\right)\right)\right)\le2\varepsilon.
\]
Taking the differential at $h=id$ we get by Lemma \ref{lem:lem12}
that 
\begin{equation}
rk\left(\ph\left(Ad_{g}\left(x\right)\right)-Ad_{\Phi\left(g\right)}\left(\ph\left(x\right)\right)\right)\le4\varepsilon,\label{eq:1}
\end{equation}
for any $x\in\mf g$. Taking again the differential at $g=id$, 
\[
rk\left(\ph\left(ad\left(y\right)\left(x\right)\right)-ad\left(\ph\left(y\right),\ph\left(x\right)\right)\right)\le8\varepsilon.
\]
Thus, 
\[
rk\left(\ph\left(\left[y,x\right]\right)-\left[\ph\left(y\right),\ph\left(x\right)\right]\right)\le8\varepsilon.
\]
The second assertion follows by applying again Lemma \ref{lem:lem12}
and differentiating $\Phi\left(g\right)-\Psi\left(g\right)$.
\end{proof}
A natural way to construct almost representations of a Lie group $G$
is by compressing representations as in Lemma \ref{lem:compression}.
It is straightforward to check that taking differentials commutes
with compressing, so that the differential of a compression of a representation
$\Phi$ of $G$ at $id$ will be a compression of the differential
$d\Phi$ on $\mf g$. 

It therefore follows from Corollary \ref{cor:real} and Proposition
\ref{prop:diff}, together with Proposition \ref{prop:rigidity},
\begin{cor}
\label{cor:Liegroups}Semisimple Lie groups are not strictly $\CC$-stable.
\end{cor}

\begin{proof}
Assume $G$ is a complex Lie group, the real case follows from Corollary
\ref{cor:real}. A codimension $1$ compression of a highest weight
representation $\Phi_{n}:G\to GL_{k_{n}}\left(\CC\right)$ is a $\frac{2}{k_{n}-1}$-representation
by Lemma \ref{lem:compression}. Take a sequence $\Phi_{n}$ of such
representations with $k_{n}\to\infty$, and let $\check{\Phi}_{n}$
be the corresponding compression. If $G$ were strictly stable, there
would be a sequence $\Psi_{n}:G\to GL_{k_{n}}\left(\CC\right)$ of
representations with $d^{strict}\left(\check{\Phi}_{n},\Psi_{n}\right)\to0$.
Taking differentials, it follows from Proposition \ref{prop:diff}
that $d^{strict}\left(d\check{\Phi}_{n},d\Psi_{n}\right)\to0$. But
this is a contradiction to Proposition \ref{prop:rigidity}, whereby
$d\check{\Phi}_{n}$ is bounded in rank (in the strict sense) from
any representation of $\mf g$.
\end{proof}
We would like to apply this result to lattices in higher rank. Since
$rk\left(\bullet\right)$ is a polynomial function on the entries
of a matrix, two maps that are close on a Zariski dense subgroup are
close on the entire group:
\begin{lem}
\label{prop:zariski}Let $G\le SL_{r}\left(\RR\right)$ be a linear
Lie group and $H\le G$ a Zariski dense subgroup. Suppose $\Phi,\Psi:G\to GL_{n}\left(\CC\right)$
are rational maps satisfying 
\[
\underset{h\in H}{sup}\,rk\left(\Phi\left(h\right)-\Psi\left(h\right)\right)\le\varepsilon.
\]
Then 
\[
\underset{g\in G}{sup}\,rk\left(\Phi\left(g\right)-\Psi\left(g\right)\right)\le\varepsilon.
\]
\end{lem}

\begin{proof}
The condition $rk\left(\Phi\left(g\right)-\Psi\left(g\right)\right)\le\varepsilon$
is polynomial in the matrix entries, and is equivalent to the vanishing
of all $\left(k+1\right)$-minors for $k=\left\lfloor \varepsilon n\right\rfloor $.
Since $\Phi,\Psi$ are rational, it follows from Zariski density of
$H$ that $\underset{g\in G}{sup}\,rk\left(\Phi\left(g\right)-\Psi\left(g\right)\right)\le\varepsilon$.
\end{proof}
We may now use Borel Density and apply Lemma \ref{prop:zariski} to
lattices in semisimple Lie groups.
\begin{cor}
\label{cor:LatticeZariski}Let $G\le SL_{r}\left(\RR\right)$ be a
connected linear semisimple Lie group and $\Gamma\le G$ a lattice
that projects densely into the maximal compact factor of $G$. Suppose
$\Phi,\Psi:G\to GL_{n}\left(\CC\right)$ are rational maps satisfying
\[
\underset{\gamma\in\Gamma}{sup}\,rk\left(\Phi\left(\gamma\right)-\Psi\left(\gamma\right)\right)\le\varepsilon.
\]
Then 
\[
\underset{g\in G}{sup}\,rk\left(\Phi\left(g\right)-\Psi\left(g\right)\right)\le\varepsilon.
\]
\end{cor}

\medskip{}
Margulis Superrigidity then implies that lattices in higher rank are
not strictly stable.
\begin{proof}[proof of Theorem \ref{thm:Lattices}]
 By adding compact factors to $G$, we may assume WLOG that $\Gamma$
is commensurable to $G_{\ZZ}$, and that it projects densely onto
the maximal compact factor of $G$. 

A codimension $1$ compression of a highest weight representation
$\Phi:G\to GL_{n}\left(\CC\right)$ is a $\frac{2}{n-1}$-representation
by Lemma \ref{lem:compression}. Let $\ph:\Gamma\to GL_{n-1}\left(\CC\right)$
be the restriction of $\Phi$ to the lattice $\Gamma$. Suppose $\psi:\Gamma\to GL_{n-1}\left(\CC\right)$
is a representation such that $d^{strict}\left(\ph,\psi\right)=\delta$.
Then by Margulis Superrigidity, restricting to a finite index subgroup
$\bar{\Gamma}\le\Gamma$, $\psi$ is the restriction of a representation
$\Psi$ of $G$. By Corollary \ref{cor:LatticeZariski}, since $\Phi,\Psi$
are rational and since $\bar{\Gamma}$ is also a lattice that projects
densely into the maximal compact factor, we have 
\[
d^{strict}\left(\Phi,\Psi\right)\le\delta,
\]
contradicting Corollary \ref{cor:Liegroups}.
\end{proof}

\subsection{Remarks on lifting almost representations of Lie algebras\label{subsec:Remarks}}

It is natural to ask whether the converse of Proposition \ref{prop:diff}
holds, i.e. whether almost representations of real Lie algebras always
arise as differentials of almost representations of the corresponding
simply connected Lie group. In particular, whether the $\varepsilon_{n}$-representations
$\ph_{n}^{\lambda}$ in Theorem \ref{thm:thm9}, restricted to some
real Lie algebra $\mf g$, arise as differentials of almost representations
of the corresponding Lie group $G$. Indeed, a positive answer for
any $\ph_{n}^{\lambda}$ for $\lambda$ such that $\chi_{\lambda}$
is not the central character of any finite dimensional representation,
would imply that Lie groups and lattices in higher rank Lie groups
are not \emph{flexibly} $\CC$-stable, using the same reasoning as
in the proof of Theorem \ref{thm:Lattices}. We do not know the answer
to this question. 

However, we do note that there seem to be some difficulties in lifting
such almost representations of $\mf g$ to those of $G$; Recall that
the Weyl group of $G$ is $W=N_{G}\left(\mf h\right)/C_{G}\left(\mf h\right)$,
so by equation \ref{eq:1} in the proof of Proposition \ref{prop:diff},
we have 
\[
rk\left(\ph\left(wxw^{-1}\right)-\Phi\left(w\right)\ph\left(x\right)\Phi\left(w\right)^{-1}\right)\le4\varepsilon,
\]
for any $x\in\mf g$ and $w\in W$. We can interpret this by saying
that $\ph$ has to be almost invariant under the Weyl group action. 

Consider the simplest case of $\mf g=\mf{sl}_{2}\left(\CC\right)$
(whose representations can be seen as the building blocks of representations
of all semisimple Lie algebras). Let 
\[
\mathbf{e}=\begin{pmatrix}0 & 1\\
0 & 0
\end{pmatrix},\quad\mathbf{f}=\begin{pmatrix}0 & 0\\
1 & 0
\end{pmatrix},\quad\mathbf{h}=\begin{pmatrix}1 & 0\\
0 & -1
\end{pmatrix}
\]
be the standard basis of $\mf g$, then the nontrivial element $w\in W$
acts on $\mf g$ by 
\[
\mathbf{h}\mapsto-\mathbf{h},\;\mathbf{e}\mapsto\mathbf{-\mathbf{f}},\;\mathbf{f}\mapsto-\mathbf{e},
\]
which when applied to $L\left(\lambda\right)$ results in the \emph{lowest
weight module }$\bar{L}\left(-\lambda\right)$ with lowest weight
$-\lambda$ (see \cite[3]{MR2567743}). When $\lambda$ is dominant
integral (in this case $\lambda\in\ZZ_{\ge0}$), these modules are
indeed isomorphic. But if $\lambda\notin\ZZ_{\ge0}$, this is not
true, but they do have the same central character. If $\ph_{n}^{\lambda}$
can be lifted to some $\Phi_{n}^{\lambda}$, this means that finite
dimensional compressions of $L\left(\lambda\right)$ are arbitrarily
close in rank to compressions of (a conjugation of) $\bar{L}\left(-\lambda\right)$.
The technique of central characters only helps in separating weight
modules with distinct central characters, but in $\mf{sl}_{2}\left(\CC\right)$
for every central character $\chi_{\lambda}$ there is a family of
weight modules 
\[
\left\{ V\left(\xi,\chi_{\lambda}\right)\st\xi\in\CC/2\ZZ\right\} 
\]
having central character $\chi_{\lambda}$ and spectrum $\xi$ (see
\cite[3.3]{MR2567743}). We do not know whether modules with the same
central character but different spectrum can be close to one another,
in the sense that their respective finite dimensional compressions
are arbitrarily close in the rank metric.

\section{Free groups }

In \cite{Rolli}, Rolli constructed uniform almost representations
of nonabelian free groups in the operator norm that are bounded from
true representations (see also \cite[Section 3]{BurgerOzawaThom}).
This inspired a construction by Becker and Chapman \cite[Section 5]{BeckerChapman}
of uniform almost homomorphisms to permutation groups $S_{n}$, that
are bounded from true representations.

We construct similar almost representations in the rank metric inspired
by Rolli's construction. It is interesting to note that it is in some
sense opposite to Rolli's construction; There, the maps $\ph$ behave
locally like representations that are contained in some $\delta$-ball
around the identity, but since $GL_{n}\left(\CC\right)$ has no small
subgroups, it is bounded from true representations where high powers
of the generators must be bounded away from the identity. In the case
of the rank metric, a finite set of elements close to the identity
will generate an algebra that is contained in a small ball around
the identity (so in a sense there are ``many'' small subgroups),
but the almost representations that we construct have elements as
far as possible from the identity.
\begin{prop}
\label{prop:free}Let $\FF_{2}=\left\langle a,b\right\rangle $ and
$n\in\NN$. Then there is a $\frac{3}{n}$-representation $\ph:\FF_{2}\to GL_{n}\left(F\right)$
such that 
\[
rk\left(\ph\left(a\right)-I_{n}\right),rk\left(\ph\left(b\right)-I_{n}\right)\le\frac{1}{n},
\]
 and such that there is an element $w\in\FF_{2}$ satisfying
\[
rk\left(\ph\left(w\right)-I_{n}\right)=1.
\]
 Moreover, for any representation $\psi:\FF_{2}\to GL_{N}\left(F\right)$,
$d^{flex}\left(\ph,\psi\right)\ge\frac{1}{6}-\frac{1}{6n}$.
\end{prop}

\begin{proof}
Let $\delta=\frac{1}{n}$ and let
\[
B_{\delta}=\left\{ A\in GL_{n}\left(F\right)\st rk\left(A-I\right)\le\delta\right\} 
\]
 be the $\delta$-ball around $I$. Let $\tau:\ZZ\to B_{\delta}$
be a symmetric map so that $\tau\left(k\right)^{-1}=\tau\left(k^{-1}\right)$,
$\tau\left(0\right)=I$, and define $\ph:\FF_{2}\to GL_{n}\left(F\right)$
by acting on a reduced word of the form
\[
w=a^{n_{1}}b^{m_{1}}\cdots a^{n_{k}}b^{m_{k}}
\]
by 
\[
\ph\left(w\right)=\tau\left(m_{1}\right)\cdots\tau\left(m_{k}\right).
\]
Allowing $n_{1}=0$ and $m_{k}=0$ defines $\ph$ on all reduced words.
Now for words $w_{1}=a^{n_{1}}b^{m_{1}}\cdots a^{n_{k}}b^{m_{k}}$,
$w_{2}=a^{\ell_{1}}b^{q_{1}}\cdots a^{\ell_{d}}b^{q_{d}}$, if $m_{k}\neq0$
and $\ell_{1}\neq0$, there is no cancellation and $\ph\left(w_{1}w_{2}\right)=\ph\left(w_{1}\right)\ph\left(w_{2}\right)$,
and the same occurs if $m_{k}=\ell_{1}=0$. If exactly one of $m_{k},\ell_{1}$
equals $0$, assume WLOG that $m_{k}=0$ and $\ell_{1}\neq0$, and
the other case is proven similarly. We have
\[
w_{1}w_{2}=a^{n_{1}}b^{m_{1}}\cdots a^{n_{k-r}}b^{m_{k-r}}ca^{\ell_{r}}b^{q_{r}}\cdots a^{\ell_{d}}b^{q_{d}}
\]
for $c$ of the form $c=a^{\alpha}b^{\beta}$, which occurs if $n_{k-i+1}=-\ell_{i}$
and $m_{k-i}=-q_{i}$ for $1\le i\le r-2$. Symmetry of $\tau$ ensures
that the corresponding powers cancel so that 
\begin{align*}
\ph\left(w_{1}\right)\ph\left(w_{2}\right) & =\tau\left(m_{1}\right)\cdots\tau\left(m_{k}\right)\tau\left(q_{1}\right)\cdots\tau\left(q_{d}\right)\\
 & =\tau\left(m_{1}\right)\cdots\tau\left(m_{k-r}\right)\tau\left(m_{k-\left(r-1\right)}\right)\tau\left(q_{r-1}\right)\tau\left(q_{r}\right)\cdots\tau\left(q_{d}\right),
\end{align*}
and 
\[
\ph\left(w_{1}w_{2}\right)=\tau\left(m_{1}\right)\cdots\tau\left(m_{k-r}\right)\tau\left(\beta\right)\tau\left(q_{r}\right)\cdots\tau\left(q_{d}\right)
\]
\[
\implies rk\left(\ph\left(w_{1}w_{2}\right)-\ph\left(w_{1}\right)\ph\left(w_{2}\right)\right)=rk\left(\tau\left(m_{k-\left(r-1\right)}\right)\tau\left(q_{r-1}\right)-\tau\left(\beta\right)\right)\le3\delta.
\]
 It is thus a uniform $3\delta$-representation. Assume further that
$\left\{ \tau\left(k\right)\right\} _{k=1}^{n}$ generate a matrix
$M$ with $rank\left(M-I_{n}\right)=n$, thus 
\[
\ph\left(abab^{2}\cdots ab^{t}\right)=\tau\left(1\right)\cdots\tau\left(t\right)=M,
\]
for some $t\in\NN$.

Let $\psi:\Gamma\to GL_{N}\left(F\right)$ be a representation and
$\varepsilon_{0}=d^{flex}\left(\ph,\psi\right)$. Since $rk\left(\ph\left(a\right)-I_{n}\right)=0$
and $rk\left(\ph\left(b\right)-I_{n}\right)\le\delta$, we have 
\begin{align*}
rk\left(\psi\left(a\right)-I_{N}\right) & \le rk\left(\psi\left(a\right)-\ph\left(a\right)\right)+rk\left(\ph\left(a\right)-I_{n}\right)+rk\left(I_{n}-I_{N}\right)\le2\varepsilon_{0},\\
rk\left(\psi\left(b\right)-I_{N}\right) & \le rk\left(\psi\left(b\right)-\ph\left(b\right)\right)+rk\left(\ph\left(b\right)-I_{n}\right)+rk\left(I_{n}-I_{N}\right)\le2\varepsilon_{0}+\delta.
\end{align*}
Since $\psi$ is a representation, any $v\in W\vcentcolon=ker\left(\psi\left(a\right)-I_{N}\right)\cap ker\left(\psi\left(b\right)-I_{N}\right)$
is fixed by any word in $\psi\left(a\right),\psi\left(b\right)$.
We have 
\begin{align*}
dim\left(W\right) & \ge N-rank\left(\psi\left(a\right)-I_{N}\right)-rank\left(\psi\left(b\right)-I_{N}\right)\\
 & \ge N-\left(4\varepsilon_{0}+\delta\right)n.
\end{align*}
 Thus for any $w\in\FF_{2}$,
\begin{align*}
rank\left(\psi\left(w\right)-I_{N}\right) & \le N-dim\left(W\right)\le\left(4\varepsilon_{0}+\delta\right)n
\end{align*}
\[
\implies rk\left(\psi\left(abab^{2}\cdots ab^{t}\right)-I_{N}\right)\le4\varepsilon_{0}+\delta.
\]
But by construction we also have 
\[
rk\left(\ph\left(abab^{2}\cdots ab^{t}\right)-I_{n}\right)=1
\]
\begin{align*}
\implies rk\left(\psi\left(abab^{2}\cdots ab^{t}\right)-I_{N}\right)\ge & 1-rk\left(\psi\left(abab^{2}\cdots ab^{t}\right)-\ph\left(abab^{2}\cdots ab^{t}\right)\right)-rk\left(I_{N}-I_{n}\right)\\
\ge & 1-2\varepsilon_{0},
\end{align*}
and hence
\[
1-2\varepsilon_{0}\le4\varepsilon_{0}+\delta
\]
\[
\implies\varepsilon_{0}\ge\frac{1}{6}-\frac{\delta}{6}=\frac{1}{6}-\frac{1}{6n}.
\]
 
\end{proof}
\begin{proof}[proof of Theorem \ref{thm:free}]
 Let $\pi:\Gamma\to\FF_{2}$ be a surjection, and choose $\gamma_{1},\gamma_{2}\in\Gamma$
such that $\pi\left(\gamma_{1}\right)=a$ and $\pi\left(\gamma_{2}\right)=b$.
Define $\tilde{\ph}\vcentcolon=\ph\circ\pi:\Gamma\to GL_{n}\left(F\right)$
where $\ph$ is as in Proposition \ref{prop:free}. Then 
\begin{align*}
rk\left(\tilde{\ph}\left(\gamma_{1}\right)-I_{n}\right) & =0\\
rk\left(\tilde{\ph}\left(\gamma_{2}\right)-I_{n}\right) & \le\frac{1}{n},
\end{align*}
and there is a word $w\in\FF_{2}$ such that 
\[
rk\left(\tilde{\ph}\left(w\left(\gamma_{1},\gamma_{2}\right)-I_{n}\right)\right)=1,
\]
where $w\left(\gamma_{1},\gamma_{2}\right)$ is the word substitution
in $\Gamma$. Arguing similarly to the proof of Proposition \ref{prop:free},
we have for any representation $\psi:\Gamma\to GL_{N}\left(F\right)$,
\[
1-2d^{flex}\left(\tilde{\ph},\psi\right)\le rk\left(w\left(\gamma_{1},\gamma_{2}\right)-I_{N}\right)\le4d^{flex}\left(\tilde{\ph},\psi\right)+\frac{1}{n}
\]
\[
\implies d^{flex}\left(\tilde{\ph},\psi\right)\ge\frac{1}{6}-\frac{1}{6n}.
\]
\end{proof}
\begin{rem}
Choosing $\left\{ \tau\left(k\right)\right\} _{k=1}^{n}$ to be transpositions
$\left(k\quad k+1\right)$ in $S_{n}$ gives another example witnessing
the instability of $\FF_{2}$ in permutations, in addition to \cite{BeckerChapman},
that is not close to representation in any field $F$, and in particular
is bounded from homomorphisms to $S_{n}$.
\end{rem}

\bibliographystyle{plain}
\bibliography{Uniform_rank_metric_stability_of_Lie_algebras_and_groups_2}

\end{document}